\documentclass[a4paper,12pt,leqno]{amsart}
\usepackage{amssymb,amsmath,amsthm}
\usepackage[mathscr]{euscript}

\newcommand{\mybibitem}[7]{\bibitem[#1]{#1} #2. {\it #3,} #4 {\bf #5}
#6 #7.}

\def\R{{\mathbb R}}
\def\Rn{{{\mathbb R}^n}}
\def\N{{\mathbb N}}

\def\E{\mathscr{E}}
\def\Di{{\mathbb D_\infty}}
\def\eps{\varepsilon}

\newtheorem{theorem}{Theorem}[section]
\newtheorem{cor}[theorem]{Corollary}

\newtheorem{proposition}[theorem]{Proposition}
\newtheorem{definition}[theorem]{Definition}
\newtheorem{remark}[theorem]{Remark}

\numberwithin{equation}{section}

\DeclareMathOperator{\BE}{BE}
\DeclareMathOperator{\CD}{CD}

\title[Characterizations of sets of finite perimeter using heat kernels]{Characterizations of sets of finite perimeter using heat kernels in metric spaces}

\author[Marola]{Niko Marola}\address[N.M.]{Department of Mathematics and Statistics, University of Helsinki, P.O. Box 68 (Gustaf H\"allstr\"omin katu 2b), FI-00014 University of Helsinki, Finland}\email{niko.marola@helsinki.fi} 

\author[Miranda jr.]{Michele Miranda jr.}\address[M.M]{Department of Mathematics and Computer Science, University of Ferrara, via Machiavelli 35, 44121, Ferrara, Italy}
\email{michele.miranda@unife.it} 

\author[Shanmugalingam]{Nageswari Shanmugalingam}\address[N.S.]{Department of Mathematical Sciences, University of Cincinnati, P.O.Box 210025, Cincinnati, OH 45221--0025, USA}\email{shanmun@uc.edu}

\begin{document}


\begin{abstract}
The overarching goal of this paper is to link the notion of sets of finite perimeter (a concept associated with $N^{1,1}$-spaces)
and the theory of heat semigroups (a concept related to $N^{1,2}$-spaces) in the setting of metric measure spaces whose 
measure is doubling and supports a $1$-Poincar\'e inequality.
We prove a characterization of sets of finite perimeter in terms of a short time behavior of the heat semigroup in 
such metric spaces. 
We also give a new characterization of ${\rm BV}$ functions in terms of a near-diagonal energy in this general setting.
\end{abstract}

\maketitle
\tableofcontents
\addtocontents{toc}{~\hfill\textbf{Page}\par}

\section{Introduction}

In this paper we study functions of bounded variation and, in particular, sets of finite perimeter on general metric 
measure spaces. More precisely, we investigate a relation between the perimeter of a set and the short-time behavior 
of the action of a heat semigroup on the characteristic function of such a set. First, we provide in Theorem~\ref{charBVDelta} a new characterization of 
${\rm BV}$ functions in terms of asymptotic behavior of a near-diagonal energy or, in other words, the near-diagonal part of 
the Korevaar--Schoen type energy \cite{KS}: Suppose that  $u\in L^1(X)$, then $u\in {\rm BV}(X)$ if, and only if, 
\[
\liminf_{\eps\to 0^+} \frac{1}{\eps}\int_{\Delta^\eps} \frac{|u(x)-u(y)|}{\sqrt{\mu(B_\eps(x))}\sqrt{{\mu(B_\eps}(y))}}\,  d\mu(x,y)
<\infty.
\]
Here $\Delta^\eps$, $\eps>0$, denotes the $\eps$-neighborhood of the diagonal in $X\times X$.

Second, we shall give in Theorem~\ref{charLedoux} a Ledoux-type characterization of sets of finite perimeter in terms of the heat semigroup \cite{L}, with the focus again on the near-diagonal part of the boundary of the set. To motivate the discussion below let us state at this point the main parts of Theorem~\ref{charLedoux}: If
\[
\liminf_{t\to 0^+} \frac{1}{\sqrt{t}}\int_{E^{\sqrt{t}}\setminus E}T_t\chi_{E}(x)\, d\mu(x)<\infty,
\]
then $E$ is of finite perimeter. On the other hand, if $E$ is a set of finite perimeter and satisfies an extra assumution in \eqref{eq:Extra}, then $E$ must satisfy the inequality above.

The characterizations obtained in this paper connect the theory of 
functions of boun\-ded variation and sets of finite perimeter to the theory of the heat semigroup and heat flow in metric spaces,
thus connecting the nonlinear potential theory associated with the index $p=1$ to the linear potential theory associated
with the index $p=2$. In the Euclidean setting such a connection is well-established, see for example~\cite{DeG}
and~\cite{L}. In the setting of Riemannian manifolds with lower bounded Ricci curvature, see~\cite{CM},
\cite{MP}, and~\cite{MP2}. In the more general metric setting there is precedence; indeed, it was shown in~\cite{JK} that when the measure on a metric measure space
$X$ is doubling, supports a $2$-Poincar\'e inequality, and satisfies
a curvature condition, a global isoperimetric inequality holds in $X$. In this note we are more interested in characterizing
sets of finite perimeter in terms of the asymptotic behavior of the heat extension of the characteristic function of the sets,
and to do so we need the stronger assumption of $1$-Poincar\'e inequality. However, we avoid the use of curvature conditions
by focusing on near-diagonal part of the set in question. It is not known whether the curvature condition
assumed in~\cite{JK}, together with the doubling property of the measure and the support of a $2$-Poincar\'e inequality,
implies the support of a $1$-Poincar\'e inequality. 

\subsection{Ledoux-type characterization}
It was observed by Ledoux in \cite{L} that 
the classical isoperimetric inequality in $\R^n$, and also in more general Gauss spaces, can be characterized by the fact 
that the $L^2$-norm of the heat semigroup acting on the characteristic function of sets is increasing under isoperimetric  
rearrangement. More precisely, condition that a set $B$ is isoperimetric, that is a minimizer of the perimeter
measure among all sets $E$ with the same measure as $B$,
is equivalent to the following $L^2$-inequality
\[
\Vert T_t\chi_E\Vert_{L^2(\R^n)} \leq \Vert T_t\chi_B\Vert_{L^2(\R^n)} 
\]
for all $t\geq 0$ and sets $E\subset\R^n$ with the same volume as the Euclidean ball $B$. Here $\chi_E$ denotes 
the characteristic function of the set $E$, and $T_t$ stands for the heat semigroup defined on $L^2(\Rn)$ by convolution with 
the classical Gauss--Weierstrass kernel so that $T_t f$
solves the Cauchy problem
\[
\left\{\begin{array}{l}
\partial_t u= \Delta u \\
u(0,x) = f(x), \end{array} \right.
\]
with $f\in L^2(\R^n)$. In particular, it was shown in \cite{L} (the reader is recommended to also see the paper \cite{P} by 
Preunkert) that
\begin{equation} \label{eq:Ledouxlimit}
\lim_{t\to 0^+} \sqrt{\frac{\pi}{t}} \int_{\R^n\setminus B} T_t\chi_B(x)\,dx = P(B),
\end{equation}
whenever $B$ is a Euclidean ball and $P(B)$ the perimeter of $B$ in $\R^n$. Moreover,  the inequality
\[
\sqrt{\frac{\pi}{t}} \int_{\R^n\setminus E} T_t\chi_E(x)\,dx \leq P(E)
\]
holds for every $t\geq 0$ and for all subsets $E$ of $\R^n$ with finite measure.

The authors in \cite{MP2} pursued the investigation of the relation between the perimeter of a set and the short-time 
behavior of the heat semigroup as described in~\eqref{eq:Ledouxlimit}. They observed, by considering the 
measure-theoretic properties of the reduced boundary 
of a set, that equality \eqref{eq:Ledouxlimit} is actually valid for all sets of finite perimeter. In addition, the finiteness 
of the limit on the left hand side in \eqref{eq:Ledouxlimit} is enough to characterize sets of finite perimeter in $\R^n$.
Similar characterizations have also been obtained in \cite{B}, \cite{D}, and \cite{P-P}, where more general convolution kernels 
than the classical Gauss--Weierstrass kernel were considered. On the other hand, the approach taken in \cite{MP2} has
 more geometric flavor. 

In the present paper, we study characterizations of sets of finite perimeter and hence, by the co-area formula, all 
functions in the ${\rm BV}$-class, in terms of the heat semigroup. The recent results in~\cite{JK} demonstrate that 
under some curvature 
assumptions on the metric measure space, if the space is also doubling in measure and supports a $2$-Poincar\'e
inequality, a (global) isoperimetric inequality follows. Adding this to the discussion above, it is reasonable to ask
whether the notion of a set of finite perimeter in a metric measure space (first defined in~\cite{Mr}), using $L^1$-approximations
by Lipschitz functions, is connected to the behavior of heat extension of the characteristic function of the set. In this paper we show that a Ledoux-type characterization of sets of finite perimeter in terms of the heat extension of the
characteristic function of such sets is possible if the measure on the space is doubling and supports a $1$-Poincar\'e inequality. 

The Ledoux characterization in the limit of \eqref{eq:Ledouxlimit} requires global information of the decay of the heat extension. So if the space is hyperbolic then the behavior of $E$ and $\R^n\setminus E$ far away from $\partial E$ also might play  a role in the limit given in \eqref{eq:Ledouxlimit}. We therefore modify the criterion and consider only regions near $\partial E$ as in \eqref{eq:Lelim}.

\subsection{De Giorgi-type characterization}
In the Euclidean setting the original definition of sets of finite perimeter in terms of heat extension, is due to De~Giorgi~\cite{DeG}. In Section~\ref{sect:DG} we consider a characterization of total variation and ${\rm BV}$ functions. Such a characterization
is related to the definition of sets of finite perimeter considered 
by De~Giorgi~\cite{DeG} making use of a regularization procedure based
on the heat kernel. Indeed, for a function $u\in L^1(\R^n)$ the limit
\begin{equation} \label{eq:DGlimit}
\lim_{t\to 0^+} \int_{\R^n} |\nabla T_tu(x)|\, dx 
\end{equation}
exists and is finite if, and only if, the distributional gradient of $u$ is an $\R^n$-valued measure $Du$ with finite total 
variation $|Du|(\R^n)$. Moreover, the limit in \eqref{eq:DGlimit} equals $|Du|(\R^n)$. This result has been generalized to 
the setting of Riemannian manifolds in \cite{CM}, with the restriction that the Ricci curvature of a manifold is bounded from
 below. We refer also to the result in \cite{MP}, where further bounds on the geometry of the manifold were
assumed. We refer also to a recent paper \cite{GP} where a more general condition on the Ricci curvature has been considered,
that is the Ricci tensor can be splitted into a sum of two terms, one which is bounded from below and the second one belonging to
a suitable Kato class. We mention, in passing, that in the setting of Carnot groups the authors of \cite{BMP} showed that the aforementioned result is 
valid  in a weaker sense, namely that both the limit inferior and superior are comparable to the total variation of $u$, but it is not  known whether equality holds.

In Section~\ref{sect:DG}, we shall give a generalization of the preceding result to metric measure spaces as discussed above by 
imposing an additional assumption on a Dirichlet form $\E$ compatible with the Cheeger differentiable 
structure and associated with $X$. We shall assume that the Dirichlet form as considered above satisfies the Bakry--\'Emery condition 
$\BE(K,\infty)$ (see Definition~\ref{def:BE}), for some $K\in\R$. Providing the analogous 
self-improvement property of the  $\BE(K,\infty)$ condition as obtained by Savar\'e in a recent paper \cite{S}, 
we obtain in Proposition~\ref{DGchar} a metric space version of the De Giorgi characterization of the total variation of a 
${\rm BV}$ function. We point out here, however, that the condition $\BE(K,\infty)$ is not satisfied by some 
Carnot groups, for example Heisenberg groups, and so our discussion does not overlap with that of \cite{BMP}. Recall  
that a $1$-Poincar\'e inequality follows from the $\BE(K,\infty)$ condition, or even from weaker curvature conditions like 
the $\CD(K,\infty)$ condition of Lott--Sturm--Villani. However, a complete separable metric space endowed with a probability 
measure and satisfying the  $\BE(K,\infty)$ condition need not be doubling. 

\subsection{Organization of the paper}
We have organized our paper as follows. In Section~\ref{sect:Prel} we recall the tools needed for our analysis in 
metric measure spaces as well as the basic properties of the heat semigroup and ${\rm BV}$ functions in this setting. 
Our main results are then stated in Theorem~\ref{charBVDelta} and Theorem~\ref{charLedoux} in 
Section~\ref{sect:BVchar} and Section~\ref{sect:Ledouxchar}, respectively. In Section~\ref{sect:DG} we 
consider a characterization of total variation and ${\rm BV}$ functions. In the appendix, see Section~\ref{sect:Appendix}, we 
gather together properties of the Bakry--\'Emery condition needed in Section~\ref{sect:DG}.

\smallskip

{\bf Acknowledgement.} 
The research was mostly carried out during the stay of
the three authors at Institut Mittag--Leffler, Sweden, in Fall~2013, and finalized while N.M. and M.M. were visiting N.S. at the 
University of Cincinnati, USA,  in Spring~2014. The authors wish to thank these both institutions for the support and kind hospitality.

N.M.~was supported by the Academy of Finland and the V\"ais\"al\"a Foundation. The research of M.M.~was supported 
partially by the PRIN 2012 and partially by FAR grant of the Department of Mathematics and Computer Science of University 
of Ferrara. N.S.~was partially funded by NSF grant \#DMS-1200915.

\section{Basic concepts}
\label{sect:Prel}

In this section we recall the basic concepts that allow for nonsmooth analysis on 
metric measure spaces.

\subsection{Standing assumptions} 
{\it Throughout the paper we will assume that $(X,d,\mu)$ is a complete metric
measure space equipped with a Borel regular doubling measure $\mu$ supporting a $1$-Poincar\'e inequality.} 

Recall that a Borel-regular measure $\mu$ is doubling if there exists a constant $c_D \geq 1$ such that for every ball 
$B_r(x):=B(x,r)= \{y\in X:\ d(x,y)<r\}$, $x\in X$ and $r>0$,
\[
0<\mu(B_{2r}(x))\leq c_D \mu(B_r(x))<\infty. 
\]
Moreover, $(X,d,\mu)$ supports a $1$-Poincar\'e inequality if 
there exist constants $c>0$ and $\lambda\geq 1$ such that for any $u\in {\rm Lip}(X)$ (real-valued 
Lipschitz continuous function on $X$), the inequality
\[
\int_{B_r(x)} |u(y)-u_{B_r(x)}|\,d\mu(y) \leq c_Pr \int_{B_{\lambda r}(x)} {\rm lip}(u)(y) \,d\mu(y)
\]
holds, where ${\rm lip}(u)$ is the local Lipschitz constant of $u$ defined as
\[
{\rm lip}(u)(y) :=\liminf_{\varrho\to 0^+} \sup_{z\in B_\varrho(y)}\frac{|u(y)-u(z)|}{d(y,z)};
\]
we shall denote by $c_P$ the minimal constant verifying the Poincar\'e inequality.
We write the integral average of a function $u$ over a ball $B_r(x)$ as $u_{B_r(x)}$. Let us mention in passing that, 
by the doubling property, for every $B_R(x)\subset X$ and $y\in B_R(x)$ and for $0<r\leq R<\infty$ the inequality
\begin{equation} \label{eq:doublingest}
\frac{\mu(B_R(x))}{\mu(B_r(y))} \leq C\left(\frac{R}{r}\right)^{q_\mu},
\end{equation}
holds, where $C$ is a positive constant depending only on $c_D$ and
$q_\mu = \log_2c_D$. In what follows, $q_\mu$ denotes a counterpart of
dimension related to the measure $\mu$ on $X$. 

Unless otherwise stated, $C$ will denote a positive constant whose
exact value is not important and it may change even within a line. A
concentric $\alpha$-dilate, $\alpha>1$, of a ball $B=B_r(x)=B(x,r)$ is
written as $\alpha B$ or $B_{\alpha r}(x)$. We write the topological
boundary of a set $E$ as $\partial E$, and $\partial^*E$ denotes the
measure theoretic boundary of $E$, i.e. the set of points $x\in X$
where both $E$ and its complement $X\setminus E$ have positive upper
density. Recall that $\partial^*E\subset \partial E$. 

Finally, by the Hausdorff measure of co-dimension one we mean the generalized
spherical Hausdorff measure $\mathcal{H}$ obtained by applying the
Carath\'eodory construction to the function $h(B_r(x)) =
\mu(B_r(x))/r$ (see \cite{A}).

\subsection{Differentiable structure} 
An upper gradient for an extended real-valued function $u$ on $X$ 
is a Borel function $g:X\to[0,\infty]$ such that
\begin{equation}\label{upperGradient}
|u(\gamma(0))-u(\gamma(l_\gamma))|\leq \int_\gamma g\, ds
\end{equation}
for every nonconstant compact rectifiable curve $\gamma:[0,l_\gamma]\to X$.
This notion is due to~\cite{HK}.
We say that $g$ is a $p$--weak upper gradient of $u$ if \eqref{upperGradient} holds for $p$--almost every curve,
see~\cite{KM}. 
If $u$ has an upper gradient in $L^p(X)$, then there exists a unique minimal $p$--weak upper gradient
$g_u\in L^p(X)$ of $u$, where $g_u\leq g$ $\mu$-a.e. for every
$p$--weak upper gradient $g\in L^p(X)$ of $u$. 

Under our standing assumptions on $X$, we have also the Cheeger
differentiable structure available (see \cite{C}) that allow us to
define a linear gradient operator for Lipschitz functions. There
exists a countable measurable partition $U_\alpha$ of $X$, and
Lipschitz coordinate charts
$X^\alpha=(X_1^\alpha,\ldots,X_{k_\alpha}^\alpha):X\to\R^{k_\alpha}$
such that $\mu(U_\alpha)>0$ for each $\alpha$, and $\mu(X\setminus
\bigcup_\alpha U_\alpha)=0$. Moreover, for all $\alpha$ the charts
$(X_1^\alpha,\ldots,X_{k_\alpha}^\alpha)$ are linearly independent on
$U_\alpha$ and $1\leq k_\alpha\leq N$, where $N$ is a constant
depending on $c_D$, $c_P$, and $\lambda$, and in addition for any
Lipschitz function $u:X\to\R$ there is an associated unique (up to a
set of zero $\mu$-measure) measurable function $D_\alpha u:U_\alpha\to
\R^{k_\alpha}$ for which the following Taylor-type approximation
\[
u(x)=u(x_0)+D_\alpha u(x_0)\cdot (X_\alpha(x)-X_\alpha(x_0))+o(d(x,x_0))
\]
holds for $\mu$-a.e. $x_0\in U_\alpha$. In particular, for $x\in
U_\alpha$ there exists a norm $\|\cdot\|_x$ on $\R^{k_\alpha}$ equivalent to
the Euclidean norm $|\cdot|$, such that $g_u(x)=\|D_\alpha u(x)\|_x$ for 
almost every $x\in U_\alpha$. Moreover, it is possible to show that there exists a constant $c>1$ such that 
$c^{-1}g_u(x) \leq |Du(x)|\leq cg_u(x)$ for all Lipschitz functions $u$ and $\mu$-a.e. $x\in X$. By $Du(x)$ we
mean $D_\alpha u(x)$ whenever $x\in U_\alpha$. Indeed, one can choose the
cover such that $U_\alpha\cap U_\beta$ is empty whenever $\alpha\ne\beta$.

For the definition of the Sobolev spaces $N^{1,p}(X)$ we will
follow \cite{Sh}. Since we assume $X$ to satisfy a $1$-Poincar\'e inequality, the  Sobolev space $N^{1,p}(X)$, 
$1\leq p<\infty$, can also be defined as the closure of the collection of Lipschitz functions on $X$ in the 
$N^{1,p}$-norm defined as
$\|u\|_{1,p}^p=\|u\|_{L^p(X)}^p+\|g_u\|_{L^p(X)}^p$. The space $N^{1,p}(X)$ equipped with the 
$N^{1,p}$-norm is a Banach space and a lattice \cite{Sh}. By \cite{FHK}, the Cheeger differentiable structure extends 
to all functions in $N^{1,p}(X)$.

\subsection{Semigroup associated with $\E$} \label{sect:Semi}

In the metric space setting, there is a standard way to construct a semigroup by 
using the Dirichlet forms approach. The best way to construct it is to use the 
$L^2$--theory of (bilinear) Dirichlet forms and then 
apply a classical result asserting that such semigroup can be extrapolated to any $L^p$, $1\leq p \leq \infty$. 

We start with the following Dirichlet form $\E:L^2(X)\times L^2(X) \to [-\infty,\infty]$ defined in terms of the Cheeger differentiable structure by
\[
\E(u,v)=\int_X Du(x) \cdot Dv(x)\, d\mu(x)
\]
with domain $D(\E):=N^{1,2}(X)$ (if $u$ or $v$ does not belong to $N^{1,2}(X)$, then 
$\E(u,v)=\infty$). This bilinear form is an example of a regular and strongly local Dirichlet form as defined in~\cite{FOT}. 
The associated infinitesimal generator $A$ acts on a dense subspace $D(A)$ of 
$N^{1,2}(X)$ so  that for each $u\in D(A)$ and for every $v\in N^{1,2}(X)$,
\[
\int_X vAu\, d\mu = -\E(u, v).
\]
The operator $A$ is dissipative in the sense that
\[
\int_X uA u\, d\mu =-\E(u,u)\leq 0
\]
and is merely the Laplacian $\Delta$ when $X=\R^n$, the Cheeger differentiable structure is the standard Euclidean structure, and $\mu$ is the Lebesgue measure.

We also point out that under our standing assumptions, the metric induced by
the form
\[
d_\E (x,y) =\sup \left\{ 
f(x)-f(y)\colon f\in N^{1,2}(X),\ |Df|\leq 1\,  \mu\mbox{-a.e.}
\right\}
\]
is bi-Lipschitz equivalent to the original metric $d$. Note that the Cheeger differential structure satisfies $|Df|$ is comparable to the minimal 2-weak upper gradient of $f$. Therefore, if $|Df|\leq 1$ $\mu$-a.e. in $X$ then $f$ has a Lipschitz continuous representative. 

\begin{remark}\label{rem:propHeat}
Associated with the above Dirichlet form $\E$ and its infinitesimal generator $A$ there 
is a Markov semigroup $(T_t)_{t>0}$ acting on $L^2(X)$ with the following properties
(we refer to \cite{FOT}, \cite{KRS}, or \cite{MMS} for properties (1) through (7)):
\begin{enumerate}
\item $T_t\circ T_s = T_{t+s}$ for all $t,s>0$;
\item since $A$ is symmetric in $L^2(X)$, then $T_t$ is self adjoint in $L^2(X)$, that is
\[
\int_X T_t f g\, d\mu =\int_X fT_t g\,d\mu,
\]
for all $f,g\in L^2(X)$;
\item $T_tf \to f$ in $L^2(X)$ when $t\to 0$;
\item since $A$ is dissipative, $T_t$ is a contraction in $L^2(X)$, i.e. $\|T_t f\|_{L^2(X)} \leq \|f\|_{L^2(X)}$ for all $f\in L^2(X)$ and $t>0$;
\item if $f\in D(A)$, then $\frac1{t}(T_t f-f) \to Af$ in $L^2(X)$ as $t\to 0$;
\item for all $t>0$ and $f\in L^2(X)$, we have that
$\partial_t T_tf$ and $AT_tf$ are both in $L^2(X)$ and
\[
\partial_t T_tf = AT_tf;
\]
\item $(T_t)_{t>0}$ is a Markovian semigroup, that is if $0\leq f\leq 1$, then $0\leq T_t f\leq 1$;
\item $(T_t)_{t>0}$ can be extended to $L^\infty(X)$ by first considering positive 
functions $f\geq0$ and sequences $f_n\in L^2(X)$, $f_n\nearrow f$ and setting
\[
T_t f=\lim_{n\to \infty} T_t f_n,
\]
we refer to \cite[p. 56]{FOT};
\item since the measure $\mu$ is doubling, $(T_t)_{t>0}$ is stochastically complete : $T_t 1=1$ \cite[Theorem 4]{St1} (see also \cite{G} for stochastic completeness in the manifold setting);
\item Since $|T_t f|\leq T_t |f|$ $\mu$-a.e. for $f\in L^1(X)\cap L^\infty(X)$, and since $L^1(X)\cap L^\infty(X)$ is dense in $L^1(X)$, $(T_t)_t$ can be extended to a contraction semigroup on $L^1(X)$.
\end{enumerate}
\end{remark}

A measurable function $p\colon\R\times X\times X \to [0,\infty]$ is called a {\it heat kernel} or {\it transition function} on $X$ associated with the semigroup $(T_t)_{t>0}$ if 
\[
T_t f(x)=\int_X p(t,x,y)f(y) \, d\mu(y),
\]
for every $f\in L^2(X)$ and for every $t>0$. For $t\leq 0$ we have $p(t,x,y) =0$, and $p(t,x,y)=p(t,y,x)$ by the symmetry 
of the semigroup. The existence of a heat kernel is a direct consequence of the linearity of the operator 
$f\mapsto T_tf$ together with the $L^\infty$-boundedness of such an operator (Markovian 
property) and the
Riesz representation theorem, see for instance Sturm~\cite[Proposition 2.3]{St2}. 
Standard regularity arguments on doubling metric measure spaces admitting a $2$-Poincar\'e inequality imply that the map $x\mapsto p(t,x,y)$
is H{\"o}lder continuous for any $(t,y)\in (0,\infty)\times X$. 

Under our standing assumptions on $X$, we have the following estimates for a heat kernel uniformly for all $x,y\in X$ 
and all $t>0$, we refer to \cite[Corollary 4.10]{St3}. There are positive constants $C,C_1,C_2$ such that
\begin{align}
  p(t,x,y) & \geq \frac{C^{-1}e^{-\frac{d(x,y)^2}{C_1 t}}}{\sqrt{\mu(B_{\sqrt{t}}(x))}\sqrt{\mu(B_{\sqrt{t}}(y))}}, \label{eq:KEL}\\
     p(t,x,y) & \le \frac{Ce^{-\frac{d(x,y)^2}{C_2 t}}}{\sqrt{\mu(B_{\sqrt{t}}(x))}\sqrt{\mu(B_{\sqrt{t}}(y))}}. \label{eq:KEU}
\end{align}

\subsection{BV functions and sets of finite perimeter}

Following \cite{Mr}, we say that a function $u\in L^1(X)$ is of bounded variation ($u\in{\rm BV}(X)$) if 
\begin{equation*}\label{totVar}
  \|Du\|(X)=\inf \left\{
    \liminf_{j\to\infty} \int_X g_{u_j}\, d\mu: u_j\in {\rm Lip}_{\textup{loc}}(X),\, u_j\to u \mbox{ in } L^1_{\textup{loc}}(X)\right\}
\end{equation*}
is finite, where $g_{u_j}$ is the minimal $1$-weak upper gradient of
$u_j$. Moreover, a Borel set $E\subset X$ is said to have finite
perimeter if $\chi_E\in {\rm BV}(X)$. We denote the perimeter measure
of $E$ by $P(E)=\|D\chi_E\|(X)$. By replacing $X$ with an open set $F$
we may define $\|Du\|(F)$ and we shall write the perimeter measure of
$E$ with respect to $F$ as $P(E,F)=\|D\chi_E\|(F)$.

Strictly speaking, the definition given in~\cite{Mr} considers ${\rm
  lip}(u_j)$ instead of $g_{u_j}$. However, under our standing
assumptions of doubling property and $1$-Poincar\'e inequality, ${\rm
  lip}(u_j) = g_{u_j}$ $\mu$-a.e. in $X$ (see~\cite{C} or~\cite{HKST}). 

An equivalent definition can be also given by way of the Cheeger
differentiable structure by replacing a $1$-weak upper gradient of
$u_j$ with its Cheeger derivative.  We shall say that $u$ has bounded
total Cheeger variation if $\|D_c u\|(X)<\infty$. A set with Cheeger
finite perimeter is a $\mu$-measurable set $E$ such that
$\|D_c\chi_E\|(X)<\infty$.  By the results contained in \cite{C}, it
follows that these two definitions are equivalent, in the sense that
$C^{-1}\|D_c u\| \leq \|D u\| \leq C\|D_c u\|$. For further equivalent definitions of $\|Du \|$, perhaps more fruitful in a metric setting 
where Poincar\'e inequality and doubling properties may not be available, see~\cite{AD}. 

It follows from a $1$-Poincar\'e inequality that for each $u\in{\rm
  BV}(X)$
\begin{equation*}
\int_{B_r(x)}|u(y)-u_{B_r(x)}|\, d\mu(y) \leq c_Pr\|Du\|(B_{2\lambda r}(x)).
\end{equation*}
The factor of $2$ in the ball on the right-hand side follows from the fact that weak limits of measures might charge the boundary 
of $B_{\lambda r}(x)$. In particular, if $E$ is a set of finite perimeter and $u=\chi_E$ we recover the following form of the 
relative  isoperimetric inequality
\begin{equation} \label{eq:IPI}
\mu(B_r(x)\cap E)\frac{\mu(B_r(x)\setminus E)}{\mu(B_r(x))} \leq c_PrP(E,B_{2\lambda r}(x)).
\end{equation}

From the proof of~\cite[Theorem~1.1]{BH} (see Section~3 of~\cite{BH})
and the above 1-Poincar\'e inequality, we have the following regularity for
sets of finite perimeter. Given a Lipschitz function $f$ on $X$, for $t\in\R$ we set $A_t=\{x\in X\, :\, f(x)>t\}$. Then for almost
every $t\in\R$ we have that 
\begin{equation}\label{Mink-Per}
 \liminf_{r\to 0} \frac{\mu\left(\bigcup_{x\in \partial A_t}B_r(x)\right)}{r}\le CP(A_t)<\infty.
\end{equation}

Next, we recall that for any $u\in {\rm BV}(X)$ and Borel set $E\subset X$ the co-area formula
\begin{equation} \label{eq:coarea}
\int_{-\infty}^\infty P(\{x\in X\colon u(x)>t\}, E)\, dt = \|Du\|(E)
\end{equation}
holds, and for the proof, we refer to \cite[Proposition 4.2]{Mr}.

Finally, we recall that the measure $P(E,F)$ is concentrated on the set
\[
\Sigma_\gamma = \left\{x:\ \limsup_{r\to 0}\min\left\{\frac{\mu(B_r(x)\cap E)}{\mu(B_r(x))}, \frac{\mu(B_r(x)\setminus E)}{\mu(B_r(x))}\right\}\geq \gamma\right\}\subset \partial^*E,
\]
where constant $\gamma>0$ depends only on $c_D, c_P$ and $\lambda$, and moreover $\mathcal{H}(\partial^*E\setminus \Sigma_\gamma)=0$ and $\mathcal{H}(\partial^*E)<\infty$. We refer the reader to \cite[Theorem~5.3, Theorem~5.4]{A}.

\section{A characterization of ${\rm BV}$ functions}
\label{sect:BVchar}

In the setting of metric measure spaces satisfying the standing assumptions listed in Section~2, we now 
give a new characterization of ${\rm BV}$ functions in terms of the near-diagonal parts of the 
Korevaar--Schoen type energy \cite{KS}. The product measure 
$\mu\otimes\mu$ in the space $X\times X$ shall be written as $\mu(x,y)$. 

We point out here that in the more general
setting of topological spaces $X$, with $X\times X$ equipped with a nonnegative symmetric Radon measure 
$\Gamma$ locally finite outside the diagonal, Maz'ya proved in~\cite{M} a conductor inequality for compactly 
supported continuous functions $f$ on $X$ for which
\[
  \langle f\rangle_{p,\Gamma}^p:=\int_{X\times X}|f(x)-f(y)|^p\, d\Gamma(x,y)<\infty.
\]
It was shown in~\cite[Theorem 2, Section~4]{M} that, when $1\le p\le q<\infty$ and $a>1$, the following 
co-area--type integral can be majorized by the energy
$\langle f\rangle_{p,\Gamma}$ 
\begin{equation}\label{eq:Mazya}
\left(\int_0^\infty {\rm cap}_{p,\Gamma}(\overline{M}_{at}, M_t)^{q/p}\, dt^q\right)^{p/q} \leq C\langle f\rangle_{p,\Gamma}^p,
\end{equation}
where $M_t = \{x\in X\colon |f(x)|>t\}$ and the capacity is obtained via minimization of 
$\langle \varphi\rangle_{p,\Gamma}^p$ over all admissible functions $\varphi$. With $p=q=1$, in general
the constant $C$ in \eqref{eq:Mazya} depends on $a$, and blows up in the order of
$(a-1)^{-1}$ as $a\to 1$, see~\cite{M}. 

If we now take $\Gamma_a$ to be defined by
\[
 d\Gamma_a(x,y)=\frac{\chi_{\Delta^{a-1}}(x,y)\, d\mu(x,y)}{(a-1)\sqrt{\mu(B_{a-1}(x))}\sqrt{\mu(B_{a-1}(y))}},
\]
we obtain a near-diagonal energy
\[
  \langle f\rangle_{1,\Gamma_a}=
     \frac{1}{a-1}\int_{\Delta^{a-1}} \frac{|f(x)-f(y)|}{\sqrt{\mu(B_{a-1}(x))}\sqrt{{\mu(B_{a-1}}(y))}}\,  d\mu(x,y),
\]
with the constant $C$ in~\eqref{eq:Mazya} now independent of $a$. 
Here $\Delta^\eps$, $\eps>0$, denotes the $\eps$-neighborhood of the diagonal in $X\times X$, i.e.
\[
\Delta^\eps =\{ (x,y)\in X\times X: d(x,y)<\eps\}.
\]
We point out here that when $p=q=1$, the above Maz'ya-type inequality, \eqref{eq:Mazya}, in this setting, is equivalent to the the following generalization of the co-area inequality (with $\eps=a-1$)
\[
   \limsup_{\eps\to 0}\frac{1}{\eps}\int_0^\infty \left(\int_{M_t}
    \int_{B_\eps(x)\setminus M_t} \frac{d\mu(y)d\mu(x)}{\sqrt{\mu(B_\eps(x))}\sqrt{\mu(B_\eps(y))}}
        \right)\, dt
     \le C\, \limsup_{\eps\to 0}\langle f\rangle_{1,\Gamma_{1+\eps}}
\]
associated with the near-diagonal Korevaar-Schoen energy.

In this section we will show that in our more specialized setting of metric measure spaces, the family (with respect
to $\eps>0$) of above energies $\langle f\rangle_{1,\Gamma_{1+\eps}}$ 
corresponding to a given function $f$ has a finite limit infimum as $\eps\to 0$
if and only if $f$ is in the ${\rm BV}$ class. The proof will also show that for small $\eps>0$ the above 
energy is controlled by a constant multiple of the ${\rm BV}$-energy of $f$. This also provides a connection 
between the energy studied
by Maz'ya~\cite{M} in our, more specialized, setting and the ${\rm BV}$-energy. 

Let us now formulate the main theorem of this section.
\begin{theorem}\label{charBVDelta}
Suppose that  $u\in L^1(X)$. Then $u\in {\rm BV}(X)$ if, and only if, 
\[
\liminf_{\eps\to 0^+} \frac{1}{\eps}\int_{\Delta^\eps} \frac{|u(x)-u(y)|}{\sqrt{\mu(B_\eps(x))}\sqrt{{\mu(B_\eps}(y))}}\,  d\mu(x,y)
<\infty.
\]
\end{theorem}

\begin{remark} \label{remark3.1}
As can be deduced from the proof, we in fact will show that if $u\in BV(X)$ or if the limit in the above theorem is finite,
then the limit \emph{supremum} of the
expression considered in the theorem is comparable with
$|Du|$, that is, there is a constant $C\ge 1$, that depends only
on the data of the space $X$, such that
\[ 
C^{-1}\Vert Du\Vert(X)\le \liminf_{\eps\to 0^+} \frac{1}{\eps} \int_{\Delta_\eps}\frac{|u(x)-u(y)|}{\sqrt{\mu(B_\eps(x))}\sqrt{\mu(B_\eps(y))}}\,d\mu(x,y)
 \le C\Vert Du\Vert(X). 
\]
\end{remark}

\begin{proof}[Proof of Theorem~\ref{charBVDelta}]
Let us assume first that $u\in {\rm BV}(X)$ and fix $\eps>0$. Then we can find a sequence of points in $X$, 
$\{x_i\}_{i\in \N}$, such that 
\[
X=\bigcup_{i\in \N} B_\eps(x_i),\ \text{ and }\ B_{\eps/2}(x_i)\cap B_{\eps/2}(x_j)=\emptyset\ \textrm{ whenever }\ i\neq j,
\]
and such that the covering has bounded overlap
\begin{equation} \label{eq:BoundedOverlap}
\sum_{i\in \N} \chi_{B_{4\lambda \eps}(x_i)}(x)\leq c_O.
\end{equation}
Since for $x\in B_\eps(x_i)$ and $y\in B_\eps(x)$ we have that $B_\eps(x_i)\subset B_{2\eps}(x)$
and $B_{\eps}(x_i)\subset B_{4\eps}(y)$, we get, by the doubling property of $\mu$,
\[
\mu(B_\eps(x_i))\leq \mu(B_{4\eps}(y))\leq c_D^2 \mu(B_\eps(y)),
\]
\[
\mu(B_\eps(x_i))\leq \mu(B_{2\eps}(x))\leq c_D \mu(B_\eps(x)).
\]
We obtain
\begin{align*}
\int_{\Delta^\eps} &\frac{|u(x)-u(y)|}{\sqrt{\mu(B_\eps(x))}\sqrt{\mu(B_\eps(y))}}\, 
d\mu(x,y) \\
\leq &
c_D^{\frac{3}{2}} \sum_{i\in \N} \frac{1}{\mu(B_\eps(x_i))} \int_{B_\eps(x_i)} 
\int_{B_\eps(x)} |u(x)-u(y)|\,d\mu(y) d\mu(x)\\
\leq &
c_D^{\frac{3}{2}} \sum_{i\in \N} \frac{1}{\mu(B_\eps(x_i))} \int_{B_\eps(x_i)} 
\int_{B_\eps(x)} |u(x)-u_{B_\eps(x_i)}|\,d\mu(y) d\mu(x) \\
& \quad +c_D^{\frac{3}{2}} \sum_{i\in \N} \frac{1}{\mu(B_\eps(x_i))} \int_{B_\eps(x_i)}
\int_{B_\eps(x)} |u(y)-u_{B_\eps(x_i)}|\,d\mu(y) d\mu(x) \\
=&C(I_1+I_2).
\end{align*}
Since $B_\eps(x)\subset B_{2\eps}(x_i)$ for $x\in B_\eps(x_i)$ we get, by a $1$-Poincar\'e inequality,
\begin{align*}
I_1=& \sum_{i\in \N} \frac{1}{\mu(B_\eps(x_i))} \int_{B_\eps(x_i)}
\mu(B_\eps(x)) |u(x)-u_{B_\eps(x_i)}|\,d\mu(x) \\
\leq &
\sum_{i\in \N} \frac{\mu(B_{2\eps}(x_i))}{\mu(B_\eps(x_i))} 
\int_{B_\eps(x_i)} |u(x)-u_{B_\eps(x_i)}|\,d\mu(x) \\
\leq & \eps c_P c_D \sum_{i\in \N}  \|Du\|(B_{2\lambda \eps}(x_i)) \leq \eps C\|Du\|(X).
\end{align*}
We treat the term $I_2$ in a similar fashion
\begin{align*}
I_2\leq&
\sum_{i\in \N} \frac{1}{\mu(B_\eps(x_i))} \int_{B_\eps(x_i)}\,d\mu(x)\int_{B_{2\eps}(x_i)} |u(y)-u_{B_\eps(x_i)}|\, d\mu(y) \\
=& \sum_{i\in \N} 
\int_{B_{2\eps}(x_i)} |u(y)-u_{B_\eps(x_i)}|\,d\mu(y) \\
\leq & \sum_{i\in \N} 
\int_{B_{2\eps}(x_i)} |u(y)-u_{B_{2\eps}(x_i)}|\,d\mu(y) \\
& \quad + \sum_{i\in \N} 
|u_{B_\eps(x_i)}-u_{B_{2\eps}(x_i)}| \mu(B_{2\eps}(x_i))
\\
\leq&
2\eps c_P(1+c_D) \sum_{i\in \N} \|Du\|(B_{4\lambda\eps}(x_i))
\leq \eps C\|Du\|(X),
\end{align*}
where we have used a $1$-Poincar\'e inequality and the fact that
\begin{align*}
\mu(B_{2\eps}(x_i)) |u_{B_\eps(x_i)}-u_{B_{2\eps}(x_i)}|\leq&
C \int_{B_\eps(x_i)} |u(x)-u_{B_{2\eps}(x_i)}|\,d\mu(x) \\
\leq&
C\int_{B_{2\eps}(x_i)} |u(x)-u_{B_{2\eps}(x_i)}|\,d\mu(x) \\
\leq & \eps C\|Du\|(B_{4\lambda\eps}(x_i)).
\end{align*}
This completes the proof of the claim that if $u\in{\rm BV}(X)$ then the limit infimum is finite (in fact, we have obtained that the $\limsup$ is finite).

For the converse statement, we assume that
\[
\liminf_{\eps\to 0^+} \frac{1}{\eps}\int_{\Delta^\eps} \frac{|u(x)-u(y)|}{\sqrt{\mu(B_\eps(x))}\sqrt{{\mu(B_\eps}(y))}}\,  d\mu(x,y)
<\infty.
\]
For every $\eps>0$, let us define a positive and finite measure on $X$ as
\begin{equation}\label{eq:meas-def}
d\mu_\eps(x) = \widetilde u_{B_\eps}(x)\,d\mu(x),
\end{equation}
where for each $x\in X$ we set
\[
\widetilde u_{B_\eps}(x) = \frac{1}{\mu(B_\eps(x))}\int_{B_\eps(x)}\frac{|u(x)-u(y)|}{\eps}\, d\mu(y).
\]
Note that 
\[ 
\int_X\tilde{u}_\eps(x)\, d\mu(x)
\le \frac{C}{\eps}\int_{\Delta_\eps} \frac{|u(x)-u(y)|}
{\sqrt{\mu(B_\eps(x))}{\sqrt{\mu(B_\eps(y))}}}\, d\mu(x,y), 
\]
and so we have that 
\begin{equation} \label{eq:Rev}
\lim_{\eps\to 0^+}\int_X\tilde{u}_\eps(x)\, d\mu(x)\le 
\liminf_{\eps\to 0^+}\frac{C}{\eps}
\int_{\Delta_\eps} \frac{|u(x)-u(y)|}
{\sqrt{\mu(B_\eps(x))}{\sqrt{\mu(B_\eps(y))}}}\, d\mu(x,y).
\end{equation}

Let $\{B_i^\eps=B_\eps(x_i)\}$ be a family of balls which covers  $X$ with the following bounded overlap property:
\[
\sum_{i\in \N} \chi_{B_{6\eps}(x_i)}(x)\leq c_O.
\]
Then, with respect to this family of balls, there exists a collection $(\varphi_i^\eps)_i$ of functions on $X$ such that each $\varphi_i^\eps$ is $C\eps^{-1}$-Lipschitz continuous, where $C$ is a constant depending only on the doubling constant $c_D$, $0\leq \varphi_i^\eps\leq 1$ for all $i$, $\varphi_i^\eps(x) = 0$ for $x\in X\setminus B_i^{2\eps}$ for every $i$, and $\sum_i\varphi_i^\eps(x) =1$ for all $x\in X$. The collection $(\varphi_i^\eps)_i$ forms a partition of unity of $X$ with respect to $\{B_i^\eps\}$. We refer the reader to \cite{HKT}, and the references therein, for these properties of  $\varphi_i^\eps$. Let us then define
\[
u_\eps(x) = \sum_{i=1}^\infty  u_{B_i^\eps}\varphi_i^\eps(x)
\]
for every $x\in X$. The function $u_\eps$ is locally Lipschitz continuous, since it is locally a finite sum of Lipschitz functions.
Furthermore, the sequence $(u_\eps)_\eps$ converges to $u$ in $L^1(X)$ as $\eps\to 0$ \cite[Lemma 5.3.(2)]{HKT}. 

We need an estimate on
the Lipschitz constant of $u_\eps$. Suppose that $x,y\in B_i^\eps$ and let $J=\{j\colon B_i^{2\eps}\cap B_j^{2\eps} \neq \emptyset\}$. By the doubling property, balls $B_i^\eps$ have bounded overlap and so the cardinality of $J$ is bounded by a constant depending only 
on the doubling constant $c_D$, that is $\# J\leq c_\#$ with $c_\#=c_\#(c_D)$.
In addition $B_j^\eps \subset B_i^{6\eps}$. By the properties of the partition of unity, 
we have exactly as in \cite{HKT} that for any $x,y\in B_i^\eps$
\begin{align*}
\vert u_\eps(x)- u_\eps(y)\vert & 
 = \left|\sum_{j\in J}(u_{{B_j^\eps}}-u_{{B_i^\eps}})(\varphi_j^\eps(x)-\varphi_j^\eps(y))\right| \\
& \leq C\frac{d(x,y)}{\eps}\sum_{j\in J} |u_{{B^\eps_i}}-u_{{B^\eps_j}}|.
\end{align*}
For every $j\in J$ we also have that
\begin{align*}
|u_{B^\eps_i}-u_{B^\eps_j}|\leq&
\frac{1}{\mu(B^\eps_i)\mu(B^\eps_j)} \int_{B^\eps_i}
\int_{B^\eps_j} |u(y)-u(x)|\, d\mu(y) d\mu(x)  \\
\leq&
\frac{1}{\mu(B^\eps_i)\mu(B^\eps_j)} \int_{B^\eps_i} \int_{B_{6\eps}(x)} |u(y)-u(x)|\, d\mu(y)\, d\mu(x)  \\
\leq &\frac{\varepsilon C}{\mu(B^\eps_i)} \mu_{6\eps}(B^\eps_i),
\end{align*}
where in the last inequality we have used the fact that $B_j^\eps\subset B_{6\eps}(x)$ whenever $x\in B^\eps_i$. In other words, we have proved that 
\[
{\rm lip}(u_\eps)(x)\leq C\frac{c_\#}{\mu(B^\eps_i)}\mu_{6\eps}(B_i^\eps)
\]
for every $x\in B_i^\eps$. Here, $\mu_{6\eps}$ is as in~\eqref{eq:meas-def}. We therefore obtain
\begin{align*}
\int_X {\rm lip}(u_\eps)(x)\, d\mu(x)  & \leq
\sum_{i\in \N} \int_{B^\eps_i} {\rm lip}(u_\eps)(x)\, d\mu(x) \leq 
C\sum_{i=1}^\infty c_\#  \mu_{6\eps}(B_i^\eps) \\
& \leq C\mu_{6\eps}(X).
\end{align*}
Now, by the assumption and the doubling property of the measure $\mu$, we have by \eqref{eq:Rev}
\begin{align*}
\liminf_{\eps\to 0^+} \mu_\eps(X) & =
\liminf_{\eps\to 0^+} \int_X \tilde u_{B_\eps}(x)\,d\mu(x) \\
& \le C \liminf_{\eps\to 0^+}\frac{1}{\eps}
\int_{\Delta_\eps}\frac{|u(x)-u(y)|}{\sqrt{
\mu(B_\eps(x))}{\sqrt{\mu(B_\eps(y))}}}
d\mu(y,x) <\infty,
\end{align*}
and hence we can find a sequence $\{\eps_j\}_j$ going to $0$ such that 
\[
\sup_{j\in \N}\mu_{\eps_j}(X)<\infty, 
\]
and then the sequence of Lipschitz functions $u_j=u_{\eps_j/6}$ converges to $u$ in $L^1(X)$ and 
\begin{align*}
\Vert Du\Vert(X) & \leq \limsup_{j\to\infty} \int_X {\rm lip}(u_j)(x)\,d\mu(x) \\
& \le C \liminf_{\eps\to 0^+}\frac{1}{\eps}
\int_{\Delta_\eps}\frac{|u(x)-u(y)|}{\sqrt{
\mu(B_\eps(x))}{\sqrt{\mu(B_\eps(y))}}}
\,d\mu(y,x) <\infty,
\end{align*}
which implies that $u\in {\rm BV}(X)$.
\end{proof} 

\section{Sets of finite perimeter: A Ledoux type characterization}
\label{sect:Ledouxchar}

We shall make use of Theorem~\ref{charBVDelta} to connect the sets of finite perimeter to the theory of heat semigroup in the sense of Ledoux~\cite{L}. The reader could also consult \cite{P}. 

The Ledoux characterization in the limit of \eqref{eq:Ledouxlimit} requires global information of the decay of the heat extension. So if the space is hyperbolic then the behavior of $E$ and $\R^n\setminus E$ far away from $\partial E$ also might play  a role in the limit given in \eqref{eq:Ledouxlimit}. We therefore modify the criterion and consider only regions near $\partial E$.

\begin{theorem} \label{charLedoux}
Let $E\subset X$ be $\mu$-measurable and assume that $E$ has finite measure. Then we have the following.
\begin{enumerate}
\item If 
\begin{equation} \label{eq:Lelim}
   \liminf_{t\to 0^+} \frac{1}{\sqrt{t}}\int_{E^{\sqrt{t}}\setminus E}T_t\chi_{E}(x)\, d\mu(x)<\infty,
\end{equation}
then $E$ is of finite perimeter.
\item If $E$ is of finite perimeter and satisfies the inequality 
\begin{equation} \label{eq:Extra}
 \liminf_{r\to 0} \frac{\mu\left(\bigcup_{x\in \partial E}B_r(x)\right)}{r}\le CP(E),
\end{equation}
then $E$ satisfies~\eqref{eq:Lelim}.
\end{enumerate}
Here, $E^{\sqrt{t}}$ is the tubular neighborhood
\[
E^{\sqrt{t}} = \bigcup_{x\in E}B_{\sqrt{t}}(x).
\]
Furthermore, every set $E$ of finite perimeter can be approximated in the relaxed sense by open sets of finite perimeter that
satisfy inequality~\eqref{eq:Extra}, that is, we can find a sequence of sets $E_k$ of finite perimeter satisfying~\eqref{eq:Extra}
such that $\chi_{E_k}\to\chi_E$ in $L^1(X)$ and 
\[
 P(E)\le \lim_{k\to\infty} \liminf_{t\to 0^+} \frac{1}{\sqrt{t}}\int_{E^{\sqrt{t}}\setminus E}T_t\chi_{E_k}(x)\, d\mu(x)
   \le CP(E).
\]
\end{theorem}

We remark that unlike in the case of Theorem 3.1, see Remark~\ref{remark3.1}, here we do not have 
equivalence between the expression in
(4.1) and the perimeter $P(E)$ of $E$. We only obtain that 
\[ 
P(E)\le C \liminf_{t\to 0^+} \frac{1}{\sqrt{t}} \int_{E^{\sqrt{t}}\setminus
E} T_t \chi_E(x)\, d\mu(x).
\]

\begin{proof}[Proof of Theorem~\ref{charLedoux}]
Note first that 
\begin{align*}
\int_X\vert T_t & \chi_E(x)-\chi_E(x)\vert\, d\mu(x) \\
=&  \int_{X}\left\vert\int_X p(t,x,y)( \chi_E(y)-\chi_E(x))\,d\mu(y)\right\vert\, d\mu(x) \\
=& \int_E\left\vert -\int_{X\setminus E} p(t,x,y)\, d\mu(y)\right\vert\, d\mu(x) \\
& \qquad \quad +\int_{X\setminus E}\left\vert\int_E p(t,x,y)\, d\mu(y)\right\vert\, d\mu(x) \\
= & 2\int_E\int_{X\setminus E}p(t,x,y)\, d\mu(x) d\mu(y) \\ 
=& 2\int_{X}\int_{X}\chi_{X\setminus E}(x)p(t,x,y)\chi_E(y)\, d\mu(x) d\mu(y) \\  
      =&2\int_{X\setminus E}T_t\chi_E(x)\,d\mu(x).
\end{align*}
Hence by the fact that $T_tf\to f$ in $L^1(X)$ as $t\to 0$ for each $f\in L^1(X)$, we have that
\[
    \lim_{t\to 0^+} \int_{X\setminus E}T_t\chi_E(x)\,d\mu(x)=0.
\]
Suppose now that in addition, 
\[
\liminf_{t\to0^+}\frac{1}{\sqrt{t}}\int_{E^{\sqrt{t}}\setminus E}T_t\chi_E(x)\, d\mu(x) <\infty.
\]
We want to show that the set $E$ has finite perimeter. By the symmetry of the heat kernel
\begin{align*}
2\int_E\int_{E^{\sqrt{t}}\setminus E} & p(t,x,y)\, d\mu(y)\, d\mu(x)
  \geq \int_E \int_{B_{\sqrt{t}}(x)\setminus E} p(t,x,y)\, d\mu(y)d\mu(x) \\
           &\qquad \qquad +\int_{E^{\sqrt{t}}\setminus E}\int_{B_{\sqrt{t}}(x)\cap E}p(t,x,y)\, d\mu(y)\, d\mu(x)\\
  &= \int_{\Delta^{\sqrt{t}}}\, p(t,x,y)|\chi_E(y)-\chi_E(x)| \, d\mu\times \mu(y,x).
\end{align*}
To obtain the last equality above, we used
the fact that $|\chi_E(x)-\chi_E(y)|=1$ precisely
when
\[
(x,y) \in E\times (X\setminus E)\cup (X\setminus E)\times E,
\] 
and zero otherwise.

By estimate \eqref{eq:KEL} for the kernel $p(t,x,y)$,
\begin{align*}
 \frac{1}{\sqrt{t}} \int_E\int_{E^{\sqrt{t}}\setminus E} & p(t,x,y)\, d\mu(y)d\mu(x) \\
 & \quad \geq \frac{C}{\sqrt{t}} 
       \int_{\Delta^{\sqrt{t}}\cap[(E^{\sqrt{t}}\cap E)\times (E^{\sqrt{t}}\setminus E)]}\frac{|\chi_E(x)-\chi_E(y)|}{\sqrt{\mu(B_{\sqrt{t}}(x))}\sqrt{\mu(B_{\sqrt{t}}(y))}}\, d\mu(y)d\mu(x) \\
    & \quad = \frac{C}{\sqrt{t}} 
       \int_{\Delta^{\sqrt{t}}}\frac{|\chi_E(x)-\chi_E(y)|}{\sqrt{\mu(B_{\sqrt{t}}(x))}\sqrt{\mu(B_{\sqrt{t}}(y))}}\, d\mu(y)d\mu(x),
\end{align*}
and so we have 
\[
  \liminf_{t\to 0^+}\frac{1}{\sqrt{t}}
    \int_{\Delta^{\sqrt{t}}}\frac{|\chi_E(x)-\chi_E(y)|}{\sqrt{\mu(B_{\sqrt{t}}(x))}\sqrt{\mu(B_{\sqrt{t}}(y))}}\, d\mu(y)d\mu(x)<\infty.
\]
By Theorem~\ref{charBVDelta},  we conclude that $\chi_E\in {\rm BV}(X)$. This completes the proof of statement~(1) of the theorem.

\smallskip

To prove statement~(2) of the theorem, suppose that $E\subset X$ is open and of finite perimeter, with
$\mu(E)$ finite, and that~\eqref{eq:Extra} is satisfied by $E$.

For each $t>0$ let 
\[
E^{\sqrt{t}}\setminus E \subset \bigcup_{k\ge 0}B_k,
\]
where $B_k := B_{2\sqrt{t}}(x_k)$ with $x_k\in \partial E$ 
such that the dilated balls $2\lambda B_k$ have bounded overlap~\eqref{eq:BoundedOverlap}, the
bound $c_O$ of the overlap depending solely on the doubling constant
of the measure $\mu$. For each $k\ge 0$, we also write $E=E_k^1\cup E_k^2$, where
\[
 E_k^1 = E\cap B_k, \quad E_k^2=E\setminus B_k. 
\]
Then
\begin{align*}
\frac{1}{\sqrt{t}}\int_{E^{\sqrt{t}}\setminus E} T_t\chi_E(x)\, d\mu(x) & = \frac{1}{\sqrt{t}}\int_E\int_{E^{\sqrt{t}}\setminus E}p(t,x,y)\, d\mu(y)d\mu(x) \\
& \leq \frac{1}{\sqrt{t}}\sum_{k=0}^\infty\int_E\int_{B_k\setminus E}p(t,x,y)\, d\mu(y)d\mu(x) \\
& \leq \frac{1}{\sqrt{t}}\sum_{k=0}^\infty\left(\int_{E_k^1}\int_{B_k \setminus E}p(t,x,y)\, d\mu(y)d\mu(x) \right.\\
& \qquad \qquad \left. + \int_{E_k^2}\int_{B_k\setminus E}p(t,x,y)\, d\mu(y)d\mu(x) \right) \\
& =: \frac{1}{\sqrt{t}}\sum_{k=0}^\infty(I_k^1 + I_k^2).
\end{align*}
Let us estimate the preceding terms separately, starting with the term $I_k^1$. 
By~\eqref{eq:KEU}, the doubling condition and \eqref{eq:doublingest}, and finally by \eqref{eq:IPI}, we have 
\begin{align*}
I_k^1 & = \int_{E\cap B_k}\int_{B_k\setminus E}p(t,x,y)\, d\mu(y)d\mu(x) \leq \int_{E\cap B_k}\int_{B_k\setminus E}\frac{C}{\mu(B_{\sqrt{t}}(x))}\,d\mu(y)d\mu(x) \\
& \leq  C\frac{\mu(B_k\setminus E)}{\mu(B_k)}\mu(E\cap B_k) \leq C\sqrt{t} P(E,2\lambda B_k).
\end{align*}
%
Therefore, we may conclude that
\[
  \frac{1}{\sqrt{t}}\sum_{k=0}^\infty I_k^1\le C P(E).
\]
We point out here again that $C$ represents constants that depend only
on the data of $X$ and whose particular value we do not care about,
and that the value of $C$ could change even within a line.

The second term $I_k^2$ is treated as follows.  We set 
\[
E_k^2 = \bigcup_{j\geq 1}[\overline{B}_{2^{j+1}\sqrt{t}}(x_k)\setminus B_{2^{j}\sqrt{t}}(x_k)]\cap E =: \bigcup_{j\geq 0}A_{k}^j.
\] 
By~\eqref{eq:KEU} we have
\begin{align*}
 I_k^2&=\int_{E\setminus B_k}\int_{B_k\setminus E}p(t,x,y)\, d\mu(y)d\mu(x) \\
      &\le C\int_{B_k\setminus E}\, \sum_{j=1}^\infty 
             \int_{A_k^j}\frac{e^{-C4^j}}{\sqrt{\mu(B_k)}\sqrt{\mu(B_{\sqrt{t}}(y))}}\, d\mu(y)d\mu(x)\\
       &\le C \frac{\mu(B_k\setminus E)}{\sqrt{\mu(B_k)}}\sum_{j=1}^\infty 
             \int_{A_k^j}\frac{e^{-C4^j}}{\sqrt{\mu(B_{\sqrt{t}}(y))}}\, d\mu(y).
\end{align*}
By~\eqref{eq:doublingest} we know that whenever $y\in B_{2^{j+1}\sqrt{t}}(x_k)$,
\[
  \frac{\mu(B_{\sqrt{t}}(y))}{\mu(B_{2^{j+1}\sqrt{t}}(x_k))}\ge C\left(\frac{\sqrt{t}}{2^{j+1}\sqrt{t}}\right)^{q_\mu}=C2^{-jq_\mu}.
\]
Therefore,
\begin{align*}
 I_k^2&\le C\, \frac{\mu(B_k\setminus E)}{\sqrt{\mu(B_k)}}\sum_{j=1}^\infty
                  \int_{A_k^j}\frac{e^{-C4^j}2^{jq_\mu/2}}{\sqrt{\mu(B_{2^{j+1}\sqrt{t}}(x_k))}}\, d\mu(y)\\
          &\le C\, \frac{\mu(B_k\setminus E)}{\sqrt{\mu(B_k)}}\sum_{j=1}^\infty
                  \frac{e^{-C4^j}2^{jq_\mu/2}}{\sqrt{\mu(B_{2^{j+1}\sqrt{t}}(x_k))}} \mu(A_k^j)\\
         &\le C\, \frac{\mu(B_k\setminus E)}{\sqrt{\mu(B_k)}}\sum_{j=1}^\infty
                  e^{-C4^j}2^{jq_\mu/2}\sqrt{\mu(B_{2^{j+1}\sqrt{t}}(x_k))}. 
\end{align*}
By~\eqref{eq:doublingest},
\[
  \mu(B_{2^{j+1}\sqrt{t}}(x_k))\le C2^{jq_\mu}\mu(B_{\sqrt{t}}(x_k))=C2^{jq_\mu} \mu(B_k).
\]
It follows that
\[
 I_k^2\le C\frac{\mu(B_k\setminus E)}{\sqrt{\mu(B_k)}}\sum_{j=0}^\infty
                  e^{-C4^j}2^{jq_\mu}\sqrt{\mu(B_k)}
            \le C\mu(B_k\setminus E).
\]
Therefore, by the bounded overlap of the balls $B_k$,
\[
  \frac{1}{\sqrt{t}}\sum_{k=0}^{\infty} I_k^2\le C \frac{\mu(E^{\sqrt{t}}\setminus E)}{\sqrt{t}},
\]
and the fact that~\eqref{eq:Extra} holds for $E$, completes the proof of the 
statement~(2) of the theorem.

Finally, the last claim of the theorem follows from the proof of~\cite[Theorem~1.1]{BH} (see 
Section~3 of~\cite{BH}), see also \cite{KKST}. This completes the proof of the theorem.
\end{proof}

\begin{remark}
The  estimates of $\sum_{k=0}^\infty I_k^1$ and $\sum_{k=0}^\infty I_k^2$ hold also for all open 
sets $E$ of finite perimeter. However, while for all open sets $E$ of finite perimeter we do have that
\[
\liminf_{t\to 0^+}\frac{1}{\sqrt{t}}\sum_{k=0}^\infty I_k^1 \leq CP(E) < \infty,
\]
to know that
\[
\liminf_{t\to 0^+}\frac{1}{\sqrt{t}}\sum_{k=0}^\infty I_k^2 \leq CP(E),
\]
we need equation \eqref{eq:Extra}. Hence, our requirements for the set $E$ in claim~(2) of the theorem. For example, in $\R^2$ the set
\[
E = \bigcup_{k\geq 0}B_{2^{-k}}(q_k)
\]
with $\{q_k\}_{k\geq 0}$ an enumeration of $\mathbb{Q}\times\mathbb{Q}$ is an open set of finite perimeter, but $\mu(E^{\sqrt{t}}\setminus E)=\infty$ for all $t>0$.

This may merely be an artifact of our method of proof. Note that in $\mathbb{R}^n$, an alternative proof due to Bakry shows that the additional requirement \eqref{eq:Extra} is not needed to obtain the Ledoux-type characterization \eqref{eq:Lelim} of sets of finite perimeter. We do not know how to adapt the proof of Bakry in 
such a generality considered in our paper.
\end{remark}

\section{Total variation: De Giorgi characterization under Bakry--\'Emery condition}
\label{sect:DG}

We close this paper by proving a metric space version of the De~Giorgi 
characterization \cite{DeG} of the total variation of a ${\rm BV}$ function. Our approach is based 
on recent works by Amborsio, Gigli, and Savar\'e \cite{AGS} and is
essentially a consequence of \cite{S}. 

We refer to \cite{CM, MP, GP, BMP} for the De~Giorgi characterization
on Riemannian manifolds and on Carnot groups, and to \cite{AMPP} for
the same result but with the semigroup of a rather general second
order elliptic operator on domains in Euclidean spaces. All of these
results require a condition on the curvature of the spaces; these
conditions are related to the Ricci curvature in the case of
Riemannian manifolds. In \cite{MP} the requirement was on a lower
bound of Ricci curvature plus a technical requirement on a lower bound
on the volume of balls. It was pointed out in~\cite{CM} that, thanks
to the works of Bakry and \'Emery (see for instance \cite{BE}), the
same result can be obtained only with a lower bound condition on the
Ricci curvature. In \cite{GP} it has been shown that one can also
require that the Ricci tensor can be split into two parts, one bounded
below and one belonging to a Kato class. Also in this case the
De~Giorgi characterization of the total variation holds true.

We recall the definition of $\BE(K,\infty)$ condition as formulated in~\cite{BE, Ba}.

\begin{definition}[Bakry--\'Emery condition] \label{def:BE}
The Dirichlet form $\E$ given in Section~\ref{sect:Semi} in relation to the Cheeger differentiable structure is said to satisfy the 
$\BE(K,\infty)$ condition, $K\in \R$, if 
for every $f\in D(A)$ such that $Af\in N^{1,2}(X)$, we have 
\begin{equation}\label{BE1}
\frac{1}{2} \int_X |Df|^2 A\varphi \, d\mu -\int_X \varphi  Df \cdot D Af \, d\mu 
\geq K \int_X \varphi |Df|^2\, d\mu,
\end{equation}
whenever $\varphi\in D(A)\cap L^\infty(X)$ is nonnegative such that $A\varphi\in L^\infty(X)$.
\end{definition}

 We remind the reader about the connections between the $\BE(K,\infty)$ condition and the 
doubling property and a Poincar\'e inequality assumed in the previous sections. It is known that a 
$1$-Poincar\'e inequality follows from the $\BE(K,\infty)$ condition, or even from weaker curvature 
conditions like the $\CD(K,\infty)$ condition of Lott--Sturm--Villani. However, a complete 
separable metric space endowed with a probability measure and satisfying the $\BE(K,\infty)$ 
condition need not be doubling.

In the appendix, we recall some consequences of the Bakry--\'Emery condition. These 
consequences are needed in the proof of the main result of this section, Proposition~\ref{DGchar}, and 
have been investigated for instance in~\cite{S}. The setting in~\cite{S} is rather general, and hence for the convenience of
the reader we provide sketches of simplified proofs of these results in the appendix. 

The main result of this section is the following. 

\begin{proposition} \label{DGchar}
Let $u\in L^1(X)$. Suppose that
\begin{equation} \label{eq:Limsup}
\limsup_{t\to 0^+}\int_X |DT_t u|\, d\mu < \infty,
\end{equation}
then $u\in {\rm BV}(X)$. Here, by assuming that $u$ satisfies \eqref{eq:Limsup}, we are
also implicitly assuming that $T_t u \in N^{1,1}(X)$ for each $t>0$. On the other hand, whenever \eqref{eq:Limsup} holds and if a Dirichlet form $\E$ compatible with the differentiable structure satisfies the $\BE(K,\infty)$ condition, we have 
\[
\|D_c u\|(X) = \lim_{t\to 0^+}\int_X |DT_tu|\, d\mu.
\]
\end{proposition}

By the self-improvement property of the Bakry-\'Emery condition $\BE(K,\infty)$, we can easily prove Proposition~\ref{DGchar}. Of course, it would be interesting to obtain a similar result, or even the weaker version as in \cite{BMP}, without imposing the Bakry--\'Emery condition on $\E$.

\begin{proof}[Proof of Proposition~\ref{DGchar}]
The claim that ~\eqref{eq:Limsup} implies that $u\in{\rm BV}(X)$ follows immediately from the definition of ${\rm BV}(X)$ upon noticing that as $u\in L^1(X)$ we must have $T_tu\to u$ in $L^1(X)$ as $t\to 0$
(see Remark~\ref{rem:propHeat}(10)) and $T_tu\in N^{1,1}(X)$. Furthermore,
it is also immediate that
\[
  \|D_cu\|(X)\le \limsup_{t\to 0^+}\int_X|DT_tu|\, d\mu.
\]

Let $u\in {\rm BV}(X)$. We consider a sequence of Lipschitz functions $(f_j)_{j\in \N}\subset {\rm Lip}_c(X)$ such that $f_j \to u$ in $L^1(X)$ and
\[
\|D_c u\|(X) = \lim_{j\to\infty} \int_X |Df_j|\, d\mu.
\]
By the lower semicontinuity of the total variation together with Proposition~\ref{prop:SelfImp} and the fact that $(T_t)_{t>0}$ is a contraction semigroup on $L^1(X)$ (see Remark~\ref{rem:propHeat}(10)), 
we conclude
\begin{align*}
\|D_c u\|(X)\leq &
\liminf_{t\to 0^+} \int_X |DT_t u|\,d\mu 
\leq
\liminf_{t\to 0^+}\liminf_{j\to\infty} \int_X |DT_t f_j|\,d\mu \\
&\leq
\liminf_{t\to 0^+}\liminf_{j\to\infty} e^{-Kt} \int_X T_t |Df_j|\,d\mu \\
&\leq
\liminf_{t\to 0^+}\liminf_{j\to\infty} e^{-Kt} \int_X |Df_j|\,d\mu 
= \|D_c u\|(X),
\end{align*}
and hence the proof is complete.
\end{proof}

In conclusion, we note that by the analog of both the De Giorgi characterization and the Ledoux characterization
of functions of bounded variation demonstrated in this paper, even in the general metric measure space
setting (with the measure doubling and supporting a $1$-Poincar\'e inequality), the behavior of 
sets of finite perimeter is intimately connected to the heat semigroup.

\section{Appendix}
\label{sect:Appendix}

In this appendix we gather together properties associated with the Bakry--\'Emery condition in Section~\ref{sect:DG}. In the metric setting a related condition, guaranteed by the logarithmic Sobolev inequality, was first studied in \cite{KRS}.
The results in this appendix are from \cite{S} adapted to our purposes. Given that we have access to the Cheeger differential structure related to the Dirichlet
form, some of the arguments in \cite{S} are simplified in our setting.

In the sequel, the following Pazy convolution operator will play an important role:
\[
\beta_\eps f=\frac{1}{\eps} \int_0^\infty T_s f \varrho (s/\eps)\,ds.
\]
Here $\varrho\in C^\infty_c(0,\infty)$ is a nonnegative convolution kernel with
$\int_0^\infty \varrho(s)\,ds=1$. We also define the measure operator $A^*$ by setting
\begin{align*}
D(A^*)=\Big\{ f\in N^{1,2}(X) & \colon \textrm{there exists } \nu\in \mathcal{M}(X) \mbox{ such that } \\
& \quad \E (f,v)=-\int_X v\, d\nu \textrm{ for all } v\in {\rm Lip}_c(X) \Big\},
\end{align*}
and denote by $A^*f$ the measure $\nu$.
Here ${\mathcal M}(X)$ denotes the set of finite Borel measures on $X$. The integration by parts formula
\[
\int_X Df\cdot D\varphi\, d\mu = -\int_X \varphi\, dA^* f
\]
can be extended from functions $\varphi\in {\rm Lip}_c(X)$ to 
functions $\varphi\in N^{1,2}(X)\cap L^\infty(X)$ if on the right hand side
we consider the representative $\tilde \varphi\in N^{1,2}(X)$ of $\varphi$, since $A^*f$ does not charge
sets of zero Sobolev $2$-capacity. We refer to \cite{MMS} for more details. Since $\mu$ is doubling and supports a $2$-Poincar\'e inequality, and because
$X$ is assumed to be complete, it follows that Lipschitz functions with
compact support form a dense subclass of $N^{1,2}(X)$ (see \cite{Sh} for example).

We consider the function class 
\[
\Di =\{ f\in D(A)\cap {\rm Lip}_b(X) \colon Af \in N^{1,2}(X)\};
\]
as we shall see in Proposition \ref{SavResult}, if $f\in \Di$, then $|Df|^2\in D(A^*)$. We shall denote by $\Gamma^*_2(f)$ the measure
\[
\Gamma^*_2(f)=\frac{1}{2} A^*|Df|^2 -(Df\cdot DA f)\, \mu =
\Gamma^\perp_2(f)+ \gamma_2(f)\, \mu,
\]
where $\Gamma^\perp_2(f) \perp \mu$ is the singular part of the measure $\Gamma^*_2(f)$ and $\gamma_2(f)$ its absolute continuous component. 
 
We list in the following proposition the main properties of functions in $\Di$ needed
in the proof of Proposition~\ref{DGchar}. We refer to \cite{S} for the details of the 
proof, but we provide a sketch of the proof for the convenience of the reader.

\begin{proposition}\label{SavResult}
Suppose that $f\in \Di$ and that $\E$ satisfies the $\BE(K,\infty)$ condition. Then
\begin{enumerate}
\item $|Df|^2\in N^{1,2}(X)$;
\item $|Df|^2\in D(A^*)$ and for any nonnegative $\varphi \in {\rm Lip}_c(X)$,
\[
\frac{1}{2}\int_X \varphi d A^*|Df|^2\, d\mu -\int_X \varphi Df\cdot DAf \, d\mu 
\geq K \int_X \varphi |Df|^2\, d\mu;
\]
\item we have
\[
\int_X\varphi |D|Df|^2|^2\, d\mu \leq 4\int_X\varphi\left( \gamma_2(f)-K|Df|^2\right)\, d\mu
\]
for any nonnegative $\varphi \in {\rm Lip}_c(X)$.
\end{enumerate}
\end{proposition}

\begin{proof}
Let us prove (1). Since $f\in \Di$, condition $\BE(K,\infty)$ implies that
\[
\int_X |Df|^2 A\varphi \, d\mu \geq 2 \int_X (K|Df|^2+Df\cdot DAf) \varphi \,d\mu
\]
for any positive $\varphi\in D(A)$.  Then, by setting $u=|Df|^2$ and $u_\varepsilon=\beta_\varepsilon u$, taking into account the fact that for any $\varphi\in D(A)$, since $A$ commutes with the Pazy convolution $\beta_\epsilon$, there holds
\[
\int_X u_\varepsilon A\varphi \, d\mu =\int_X u A \beta_\epsilon \varphi \, d\mu.
\]
We obtain
\begin{align*}
\int_X |Du_\epsilon|^2 \,d\mu =&
-\int_X u_\epsilon A u_\epsilon \,d\mu
=-\int_X u A \beta_\epsilon u_\epsilon \,d\mu \\
\leq& 
-2 \int_X (K|Df|^2+Df\cdot DAf) \beta_\epsilon u_\epsilon \,d\mu. 
\end{align*}
Passing to the limit as $\varepsilon\to 0$, we arrive at the conclusion that
$|Df|^2\in N^{1,2}(X)$ with
\[
\int_X |D|Df|^2|^2 \,d\mu \leq -2 \int_X (K|Df|^2+Df\cdot DAf) |Df|^2 \,d\mu. 
\]

Let us then prove (2). Using the $\BE(K,\infty)$ condition and by approximating any ${\rm Lip}_c(X)$ function with functions in $D(A)$, we can deduce that
\[
\int_X Au_\varepsilon \varphi \,d\mu \geq -\int_X g\beta_\epsilon \varphi \,d\mu
=-\int_X \beta_\varepsilon g \varphi \,d\mu
\]
for any $\varphi\in {\rm Lip}_c(X)$, where we have defined
\[
g=-2 (K|Df|^2+Df\cdot DAf).
\]
Passing to the limit as $\varepsilon \to 0$, we define
\[
{\mathcal L}_f(\varphi):= -\int_X D |Df|^2\cdot D\varphi \, d\mu +
\int_X g \varphi \,d\mu
\]
which is a positive linear functional defined on ${\rm Lip}_c(X)$. If in the estimate
\[
0\leq \int_X (Au_\varepsilon+\beta_\epsilon g)\varphi \, d\mu
\]
we take an increasing sequence of Lipschitz functions $\varphi_n$ 
such that $0\leq \varphi_n\leq 1$, $\varphi_n\equiv 1$ on $B_n(x_0)$ and 
$\varphi_n\equiv 0$ on $X\setminus B_{n+1}(x_0)$ for a fixed point $x_0$, by
the dominated convergence theorem,
\begin{align*}
0\leq & \int_X (Au_\varepsilon +\beta_\varepsilon g)\varphi_n\,d\mu 
\leq \int_X (Au_\varepsilon +\beta_\varepsilon g)\,d\mu =\int_X \beta_\varepsilon gd\mu =\int_X g\,d\mu.
\end{align*}
Note that 
\[
\int_X Au_\eps\, d\mu= \int_X \chi_X Au_\eps\, d\mu=-\int_X Du\cdot D\chi_X d\mu=0
\]
since $D\chi_X=0$. We also note that
\[ 
|Au_\eps|\le \frac{1}{\eps} \int_{s_\eps}^{S_\eps}|A T_s(u)| \rho(s/\eps)\, ds \le C_\eps \int_{s_\eps}^{S_\eps} |A T_s (u)|\ ds,
\]
and, therefore, $A u_\eps$ is in $ L^1(X)$. 

We have obtained
\[
|\mathcal{L}_f(\varphi)| \leq \|\varphi\|_\infty \int_X g\,d\mu, 
\]
and then $\mathcal{L}_f$ can be represented by a positive measure $\lambda\in \mathcal{M}(X)$. This proves that $|Df|^2\in D(A^*)$ with
\[
A^*|Df|^2 =\lambda -g\mu.
\]
The positivity of the measure
\[
\frac{1}{2} A^* |Df|^2 -(Df\cdot DAf+K|Df|^2) \mu
\]
follows by the positivity of $\lambda$.

The proof of (3) is rather technical and does not simplify in our setting, 
so we refer to \cite[Theorem 3.4]{S}.
\end{proof}

The Bakry--\'Emery condition in Definition~\ref{def:BE} has equivalent 
formulations, we refer to \cite{W} for the Riemannian case and to \cite{AGS, S} for the metric space setting. We recall here some of those equivalent formulations that we shall use.

\begin{proposition}
The following are equivalent:
\begin{enumerate}
\item $\E$ satisfies the $\BE(K,\infty)$ condition;
\item for any $f\in N^{1,2}(X)$, any $t>0$, and for every nonnegative $\varphi\in{\rm Lip}_c(X)$
\begin{equation}\label{BE2}
\int_X \varphi|D T_t f|^2\, d\mu \leq e^{-2Kt}\int_X \varphi T_t|Df|^2\, d\mu;
\end{equation}

\item for any $f\in L^2(X)$, any $t>0$, and for every nonnegative $\varphi\in{\rm Lip}_c(X)$
\begin{equation}\label{BE3}
\frac{e^{2Kt}-1}{K} \int_X \varphi |DT_t f|^2\, d\mu \leq
\int_X \varphi (T_t f^2 - (T_tf)^2)\, d\mu.
\end{equation}
\end{enumerate}
\end{proposition}

\begin{proof}
Let us show that (1) implies (2). Let $f\in N^{1,2}(X)$ and $\varphi\in D(A)\cap L^\infty(X)$ such that $A\varphi\in L^\infty(X)$ and $\varphi\geq 0$ be fixed. We define 
\[
G_\varphi(s)=\int_X T_s \varphi |D T_{t-s} f|^2 \,d\mu.
\]
Taking derivatives, and by $\BE(K,\infty)$, we get
\begin{align*}
G_\varphi'(s)=&
\int_X AT_s \varphi |DT_{t-s}f|^2\,d\mu -
2\int_X T_s \varphi DT_{t-s}f\cdot DAT_{t-s} f \,d\mu
\\
\geq &
2K \int_X T_s \varphi |DT_{t-s} f|^2\,d\mu = 2K G_\varphi(s).
\end{align*}
If for some $s>0$, $G_\varphi(s)=0$, then if $\varphi$ is not 
identically zero, $T_s\varphi>0$ and we must have $|DT_{t-s}f|= 0$ at $\mu$-a.e.  Hence $T_{t-s}f$ is constant and thus $f$ is also
constant. In this case (2) is trivially satisfied.
On the other hand, if $G_\varphi(s)>0$, we can integrate and conclude that
\[
G_\varphi(t) \geq G_\varphi(0)e^{2Kt},
\]
which is condition (2).

Let us show that (2) implies (3). By introducing the function
\[
G_\varphi(s)=\int_X T_s \varphi (T_{t-s} f)^2\, d\mu,
\]
taking derivatives, and taking into account (2), we obtain
\begin{align*}
G_\varphi'(s) & = \int_X AT_s \varphi (T_{t-s}f)^2\,d\mu -2\int_X T_s \varphi T_{t-s}f AT_{t-s}f\,d\mu = 2\int_X \varphi T_s|DT_{t-s} f|^2\,d\mu \\
& \geq
2e^{2Ks} \int_X \varphi |DT_sT_{t-s}f|^2 \,d\mu 
=2e^{2Ks}\int_X \varphi  |DT_{t}f|^2 \,d\mu.
\end{align*}
By integrating  with respect to $s$, we obtain condition (3). 

\noindent 
Let us now assume (3) and let us take $f\in D(A)$ such that $Af\in N^{1,2}(X)$ and $\varphi\in D(A)\cap L^\infty(X)$ with $A\varphi \in L^\infty$. We can define the function
\[
F_{\varphi}(t) =\int_X \varphi (T_t f^2 -(T_t f)^2)\,d\mu.
\]
The second order Taylor expansion formula implies 
\[
F_\varphi(t)=2t \int_X \varphi |Df|^2 \,d\mu  +t^2\left(
\int_X |Df|^2 A\varphi \,d\mu 
+2\int_X \varphi Df\cdot DAf \,d\mu 
\right)
+o(t^2).
\]
In the same way, we obtain that
\begin{align*}
G_\varphi (t) & =\frac{e^{2Kt}-1}{K} \int_X\varphi |DT_t f|^2\,d\mu \\
& = 2t \int_X \varphi |Df|^2 d\mu +2t^2 \left(
K\int_X \varphi |Df|^2\,d\mu 
+2 \int_X \varphi Df\cdot DAf \,d\mu \right)+o(t^2).
\end{align*}
Then, since
\begin{align*}
0\leq &
F_\varphi(t)-G_\varphi(t)  \\
=&2t^2\left(
\int_X A\varphi |Df|^2 \,d\mu 
-2\int_X \varphi Df\cdot DAf \,d\mu -2K \int_X \varphi |Df|\,d\mu 
\right)+o(t^2)
\end{align*} 
we obtain condition (1).
\end{proof}

Under the above hypotheses, we can prove the following self-im\-pro\-ving result of the Bakry--\'Emery condition. The result is essentially contained in a monograph by Bakry~\cite{Bakry}, see also Deuschel and Strook \cite[Lemma 6.2.39]{DS}, in the Riemannian setting. The self-improvement of the Bakry--\'Emery condition has also been obtained by Savar\'e in \cite[Corollary 3.5]{S} in very general setting. We provide a proof here for the sake of completeness.

\begin{proposition} \label{prop:SelfImp}
Suppose that a Dirichlet form $\E$ compatible with the differentiable structure satisfies the $\BE(K,\infty)$ condition for some 
$K\in\R$. Then for any $f\in N^{1,2}(X)$, any $t>0$, and every nonnegative $\varphi\in {\rm Lip}_c(X)$, we have
\[
\int_X\varphi|DT_t f|\, d\mu \leq e^{-Kt}\int_X \varphi T_t |Df|\, d\mu.
\]
\end{proposition}

\begin{proof}
We start by considering $f\in \Di$; we point out that by \eqref{BE3}, 
$\Di$ is dense in $N^{1,2}(X)$. 
Indeed, ${\rm Lip}_b(X)$ is dense in $N^{1,2}(X)$ and taking approximations of
functions $f\in N^{1,2}(X)\cap {\rm Lip}_b(X)$ in terms of the semigroup,
we get that $T_t f\in \Di$ for any $t>0$ by \eqref{BE3} and 
\[
AT_t f =T_{t/2} A T_{t/2} f\in D(A)\subset N^{1,2}(X).
\]
Here we also used the fact that by the Bakry--\'Emery condition \eqref{BE3}, if $f\in {\rm Lip}_b(X)$ then $T_tf\in{\rm Lip}_b(X)$.

Let us fix $\delta>0$ and set
\begin{equation}\label{defudelta}
u^\delta_t(x):=\sqrt{|DT_tf(x)|^2+\delta^2}-\delta.
\end{equation} 
We also fix a nonnegative function $\varphi\in N^{1,2}(X)\cap L^\infty(X)$ and introduce
the function
\[
G_\delta(s)=\int_X T_s\varphi  u^\delta_{t-s} \,d\mu.
\]
If we take the derivative, we obtain:
\begin{align*}
G'&_\delta(s)=
\int_X AT_s \varphi u^\delta_{t-s}\,d\mu 
-\int_X \frac{T_s\varphi}{u^\delta_{t-s}+\delta} DT_{t-s} f\cdot DAT_{t-s}f \,d\mu \\
&=
-\int_X \frac{1}{2(u^\delta_{t-s}+\delta)}DT_s\varphi\cdot D |DT_{t-s}f|^2\,d\mu 
-\int_X \frac{T_s\varphi}{u^\delta_{t-s}+\delta} DT_{t-s} f\cdot DAT_{t-s}f \,d\mu \\
&=
\int_X \frac{T_s\varphi}{2(u^\delta_{t-s}+\delta)} \,dA^* |DT_{t-s}f|^2
-\frac{1}{4}\int_X \frac{T_s\varphi}{(u^\delta_{t-s}+\delta)^3} |D| DT_{t-s}f|^2|^2\,d\mu \\
& \qquad \quad -\int_X \frac{T_s\varphi}{u^\delta_{t-s}+\delta} DT_{t-s} f\cdot DAT_{t-s}f \,d\mu \\
&=
\int_X\frac{T_s\varphi}{u^\delta_{t-s}+\delta} \,d\Gamma_2^*(T_{t-s}f)
-\frac{1}{4}\int_X \frac{T_s\varphi}{(u^\delta_{t-s}+\delta)^3} |D| DT_{t-s}f|^2|^2\,d\mu \\
&\geq
\int_X \frac{T_s\varphi}{ u^\delta_{t-s}+\delta}\gamma_2(T_{t-s}f)\,d\mu\\
& \qquad \quad -\int_X \frac{T_s\varphi}{(u^\delta_{t-s}+\delta)^3} |DT_{t-s}f|^2
\Big( \gamma_2 (T_{t-s}f)-K |DT_{t-s}f|^2\Big) \,d\mu,
\end{align*}
where in the last line we have used the properties contained in Pro\-po\-si\-tion~\ref{SavResult}.  We have
\begin{align*}
G'_\delta(s)\geq &
\int_X \frac{T_s\varphi}{(u^\delta_{t-s}+\delta)^3} 
\Big(
\delta^2\gamma_2(T_{t-s}f)- K\delta^2|DT_{t-s}f|^2 \\
& \qquad \quad + K |DT_{t-s}f|^2 (u^\delta_{t-s}+\delta)^2
\Big)\,d\mu \\
\geq&
K\int_X \frac{T_s\varphi}{u^\delta_{t-s}+\delta} |DT_{t-s}f|^2\, d\mu
\geq
K\int_X T_s\varphi u^\delta_{t-s}\,d\mu,
\end{align*}
where we used the fact that $\gamma_2(f)-K|Df|^2\geq 0$ $\mu$-a.e. by Proposition~\ref{SavResult}. We thus have $G'_\delta(s)\geq K G_\delta(s)$, and by integrating this over $(0,t)$, we arrive at
\[
\int_X \varphi (\sqrt{ |DT_t f|^2 +\delta^2}-\delta)\,d\mu
\leq e^{-Kt} \int_X T_t \varphi (\sqrt{|Df|^2+\delta^2}-\delta)\,d\mu.
\]
Passing to the limit $\delta\to 0$ and using the fact that the semigroup
is self-adjoint, we finally obtain the desired inequality.
\end{proof}

\end{document}